\documentclass[12pt]{article}
\usepackage{amsmath,latexsym,amssymb}
\usepackage{enumerate}
\usepackage{amsthm}
\usepackage{mathrsfs}
\usepackage{epic,eepic}

\newcommand\RR{\mathbb{R}}

\newcommand\FF{\mathscr{F}}
\newcommand\GG{\mathscr{G}}

\newcommand\PP{\mathcal{P}}

\renewcommand{\PP}{\mathbb{P}} 
   
\newcommand{\EE}{\mathbb{E}} 
\DeclareMathOperator\sI{I}   
\newcommand{\rI}{\mathrm{I}} 
\newcommand{\sign}{\mathrm{sgn}}

\newtheorem{theorem}{Theorem}[section]
\newtheorem{definition}[theorem]{Definition}

\newtheorem{example}[theorem]{Example}
\newtheorem{examples}[theorem]{Examples}
\newtheorem{lemma}[theorem]{Lemma}

\newtheorem{proposition}[theorem]{Proposition}

\newcommand{\address}{Address: Department of Mathematics, University of North Texas, 1155 Union Circle \#311430, Denton, TX 76203-5017, USA; E-mail: allaart@unt.edu, AndrewAllen@my.unt.edu}

\title{A random walk version of Robbins' problem: small horizon}
\author{Pieter C. Allaart and Andrew Allen \footnote{\address}}

\begin{document}

\maketitle

\begin{abstract}

In Robbins' problem of minimizing the expected rank, a finite sequence of $n$ independent, identically distributed random variables are observed sequentially and the objective is to stop at such a time that the expected rank of the selected variable (among the sequence of all $n$ variables) is as small as possible. In this paper we consider an analogous problem in which the observed random variables are the steps of a symmetric random walk. Assuming continuously distributed step sizes, we describe the optimal stopping rules for the cases $n=2$ and $n=3$ in two versions of the problem: a ``full information" version in which the actual steps of the random walk are disclosed to the decision maker; and a ``partial information" version in which only the relative ranks of the positions taken by the random walk are observed. When $n=3$, the optimal rule and expected rank depend on the distribution of the step sizes. We give sharp bounds for the optimal expected rank in the partial information version, and fairly sharp bounds in the full information version.


\end{abstract}

\section{Introduction}

Let $X_1,X_2,\dots,X_n$ be a finite sequence of independent, identically distributed (i.i.d.) random variables whose common distribution is continuous and symmetric about $0$, and consider the random walk $S_n:=X_1+X_2+\dots+X_n$, $n\geq 1$, with $S_0\equiv 0$. Let 
$$R_k := \sum_{i=0}^n \sI(S_k \leq S_i)$$ 
denote the rank of $S_k$ among $S_0, S_1, \cdots, S_n$. In this paper, we are interested in finding a stopping time $\tau$ so as to minimize the expected rank $\EE(R_\tau)$. We consider two versions of the problem: In the {\em full information version}, we assume that the values of $X_1,X_2,\dots$ are observed completely, so we can use any stopping time $\tau$ adapted to the filtration $\FF_k:=\sigma(\{X_1,\dots,X_k\})=\sigma(\{S_0,S_1,\dots,S_k\}), k=0,1,\dots,n$. By contrast, in the {\em relative ranks version} of the problem, we assume that only the relative ranks
\[
\tilde{R}_k := \sum_{i=0}^k \rI(S_k \leq S_i), \qquad k=0,1,\dots,n
\]
are observed, so only stopping times adapted to the filtration $\GG_k:=\sigma(\{\tilde{R}_0,\tilde{R}_1,\dots,\tilde{R}_k\})$ may be used.

In this note, we give a complete solution to both versions of the problem when the time horizon is small ($n\leq 3$). In a forthcoming paper, we will present bounds for the optimal expected rank when $n$ is large, and consider a continuous-time version of the problem in which the random walk is replaced by a Brownian motion. As far as we are aware, these are the first works to give a detailed treatment of the expected rank problem for random walks. Although in this paper we consider the problem for just one simple stochastic model (i.e. a symmetric random walk), it should be noted that ranks are invariant under monotone increasing transformations, so the solution to our problem will also apply to, for instance, a geometric Brownian motion sampled at discrete time steps.

The problem of minimizing the expected rank arose in the 1960's as a variation of the classical secretary problem. In the traditional setup, there is a sequence $\xi_1,\dots,\xi_n$ of i.i.d. continuous random variables, and the objective is to minimize the expected rank $\EE(\rho_\tau)$, where $\rho_k:=\sum_{i=1}^n \sI(\xi_k\leq \xi_i)$, $k=1,\dots,n$. The relative ranks version of this problem, in which only the random variables $\tilde{\rho}_k:=\sum_{i=1}^k \sI(\xi_k\leq \xi_i)$ are observed, was solved completely by Chow et al. \cite{CMRS}. They showed that, as $n\to\infty$, the optimal expected rank converges to approximately 3.87. On the other hand, the full information version of the problem, now known as {\em Robbins' problem}, remains open to this day. Despite considerable effort by various authors (e.g. \cite{Assaf-SamuelCahn,Bruss-Ferguson1,Bruss-Ferguson2}), the asymptotic expected rank is known only to lie between $1.908$ and $2.327$. For an excellent survey of what is known about Robbins' problem, see \cite{Bruss}. Recently, Dendievel and Swan \cite{Dendievel-Swan} found the exact solution of Robbins' problem for $n=4$.

It should be noted that in the expected rank problem for i.i.d. random variables, neither the optimal stopping rule nor the expected rank depends on the distribution of the $\xi_i$'s, as long as it is continuous. As will be seen below, this is no longer the case for the random walk version of the problem. For $n\geq 3$, both the optimal rule and the optimal expected rank depend on the distribution of the steps $X_1,X_2,\dots$ of the walk in both the full information version and the relative ranks version of the problem, though less so in the latter.

\section{Results} \label{sec:results}

Let $F$ denote the common distribution function of $X_1,X_2,\dots$. If $n=1$, it does not matter whether we stop at time $0$ or time $1$: in either case our expected rank is $3/2$ by the symmetry of $X_1$.
For a random walk with two steps, the problem is still distribution-invariant:

\begin{theorem} \label{thm:two-step}
Let $n=2$. Then the optimal rule in both the full information and the relative rank versions of the problem is
\[
\tau^*=\begin{cases}
1 & \mbox{if $S_1>0$},\\
2 & \mbox{otherwise},
\end{cases}
\]
and the minimum expected rank is $\EE(R_{\tau^*})=15/8$.
\end{theorem}

By contrast, for a random walk with three steps, the optimal rule and expected return depend on the distribution in both versions of the problem.

\begin{theorem} \label{thm:three-step-full}
Let $n=3$. In the full information version of the problem, there is a number $x_1^*>0$ (given implicitly as a solution of equation \eqref{eq:x1-equation} below) such that the optimal rule is
\begin{equation}
\tau^*=\begin{cases}
1 & \mbox{if $0<X_1\leq x_1^*$},\\
2 & \mbox{if $X_1>x_1^*$ and $X_2\in(0,\infty)\cup(-X_1,-\frac{1}{2} X_1]$; or} \\
  & \ \ \ \mbox{$X_1\leq 0$ and ${X_2 \in (0,-\frac12 X_1] \cup (-X_1, \infty)}$},\\
3 & \mbox{in all other cases}.
\end{cases}
\end{equation}
(See Figure \ref{fig:stopping-regions}.) Moreover, $F(x_1^*)\geq \frac12+\frac{\sqrt{2}}{4}\approx .85355$, and the optimal expected rank $\EE(R_{\tau^*})$ satisfies the inequalities
\begin{equation}
2.2413\leq \EE(R_{\tau^*})\leq \frac{55}{24}\approx 2.2917,
\label{eq:V-estimates}
\end{equation}
in which the upper bound is attained.
\end{theorem}

\begin{figure}
\begin{center}
\begin{picture}(220,220)(-20,-20)
\thicklines
\put(100,-5){\vector(0,1){210}}
\put(210,102){\makebox(0,0)[tl]{$x_1$}}
\put(-5,100){\vector(1,0){210}}
\put(95,217){\makebox(0,0)[tl]{$x_2$}}
\thinlines
\dashline{4}(140,0)(140,200)
\put(125,215){\makebox(0,0)[tl]{$x_1=x_1^*$}}
\put(100,100){\line(-1,1){100}}
\put(-20,215){\makebox(0,0)[tl]{$x_2=-x_1$}}
\put(100,100){\line(-2,1){100}}
\put(-55,165){\makebox(0,0)[tl]{$x_2=-x_1/2$}}
\put(140,60){\line(1,-1){60}}
\put(185,-5){\makebox(0,0)[tl]{$x_2=-x_1$}}
\put(140,80){\line(2,-1){60}}
\put(198,48){\makebox(0,0)[tl]{$x_2=-x_1/2$}}
\put(115,150){\makebox(0,0)[tl]{$1$}}
\put(115,50){\makebox(0,0)[tl]{$1$}}
\put(165,150){\makebox(0,0)[tl]{$2$}}
\put(165,88){\makebox(0,0)[tl]{$3$}}
\put(165,55){\makebox(0,0)[tl]{$2$}}
\put(160,15){\makebox(0,0)[tl]{$3$}}
\put(60,180){\makebox(0,0)[tl]{$2$}}
\put(35,150){\makebox(0,0)[tl]{$3$}}
\put(35,120){\makebox(0,0)[tl]{$2$}}
\put(45,55){\makebox(0,0)[tl]{$3$}}
\end{picture}
\caption{The stopping regions for the full-information problem with $n=3$. Each region is labeled with the corresponding value of $\tau^*$.}
\label{fig:stopping-regions}
\end{center}
\end{figure}
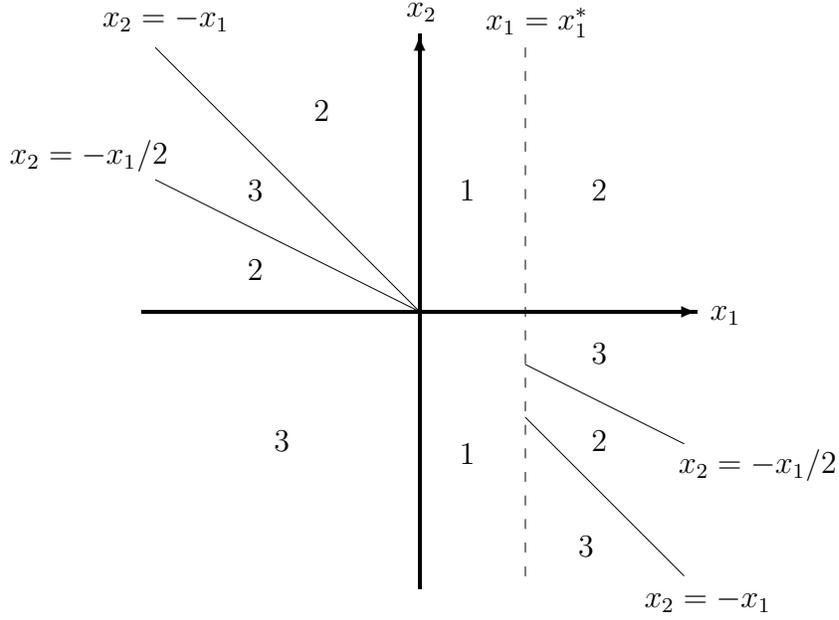


For the relative ranks version of the problem, we introduce two parameters $p=p_X$ and $q=q_X$, defined as
\begin{align*}
p:&=p_X:=\PP(0<X_1<X_2<X_3<X_1+X_2), \\ 
q:&=q_X:=\PP(0<X_1<X_2<X_1+X_2<X_3).
\end{align*}
Observe that $p+q=\PP(0<X_1<X_2<X_3)=1/48$ by the symmetry of $X$. We first give sharp bounds on $p$ for an important class of random variables.

\begin{definition}
Let $\mathscr{U}$ be the class of symmetric random variables $X$ whose distribution function $F$ is continuous and satisfies 
\begin{equation*}
F(x)-F(0)\geq F(x+y)-F(y) \qquad \forall x,y > 0.
\end{equation*}
\end{definition}
Note that all unimodal symmetric random variables are in $\mathscr{U}$, but so is, for example, the random variable that is uniform on $(-4,-3) \cup (-2,2) \cup (3,4)$.

\begin{proposition} \label{prop:p-bounds-unimodal}
Let $X\in \mathscr{U}$. Then $0<p_X\leq 1/96$. These bounds are sharp, and the upper bound is attained when $X$ has the uniform distribution on $(-1,1)$ (or any other interval symmetric about $0$).
\end{proposition}

\begin{theorem} \label{thm:three-step-relative-ranks}
Let $n=3$. In the relative ranks version of the problem, the optimal rule $\tau^*$ is as follows: 
\begin{enumerate}[a)]
\item If $p \leq q$ (in particular, if $X \in \mathscr{U}$), then 
\[
\tau^* = \begin{cases} 
      1 & \text{if $S_1 > 0$}, \\
      2 & \text{if $S_1 \leq 0$ and $S_2>S_1$}, \\
      3 & \text{otherwise}. 
   \end{cases}
\]
\item If instead $p > q$, then 
\[
\tau^* = \begin{cases} 
      1 & \text{if $S_1 > 0$}, \\
      2 & \text{if $S_1 \leq 0 < S_2$}, \\
      3 & \text{otherwise}.
   \end{cases}
\]
\end{enumerate}
The optimal expected rank is 
\[
\EE(R_{\tau ^*}) = \min\left\{\frac{55}{24}, \frac{109}{48} + 2p\right\}.
\]
Further, 
\[
2.2708\approx\frac{109}{48} < \EE(R_{\tau ^*}) \leq \frac{55}{24}\approx 2.2917, 
\]
and these bounds are sharp.
\end{theorem}

Note that the optimal rule and expected rank only depend on the distribution of $X$ through the parameter $p$, and remain constant once $p$ exceeds $1/96$.



\section{Proofs for the full information case} \label{sec:proofs-full}

We begin by making a few simple observations:
\begin{enumerate}
\item If we stop at time $k$, then our expected overall rank in either version of the problem is
\[
\EE(R_k|\FF_k)=\EE(R_k|\GG_k)=\tilde{R}_k+\frac{n-k}{2},
\]
by the symmetry of the walk.
\item When we are at a running minimum of the random walk, it is always optimal to continue. That is, if at time $0\leq k<n$ we have $\tilde{R}_k=k+1$ (or equivalently, $S_k\leq S_i$ for all $i\leq k$), then in view of the first observation above,
\begin{align*}
\EE(R_k|\FF_k)&=k+1+\frac{n-k}{2}\\
&=\frac{1}{2}\left(k+2 + \frac{n-k-1}{2}\right) + \frac{1}{2}\left(k+1 + \frac{n-k-1}{2}\right)\\
&\geq \EE(R_{k+1}|\FF_k).
\end{align*}
Thus, continuing one more step and then stopping is at least as good as stopping immediately. This holds also when we replace $\FF_k$ by $\GG_k$.
\item If we are at time $n-1$ having observed $S_1,\dots,S_{n-1}$, and if we choose to continue to the $n$th (and last) step, our expected rank is
\begin{align*}
\EE(R_n|S_1,\dots,S_{n-1})&=\sum_{i=0}^{n}\PP(S_n\leq S_i|S_1,\dots,S_{n-1})\\
&=1+\sum_{i=0}^{n-1}\PP(X_n\leq S_i-S_{n-1}|S_1,\dots,S_{n-1})\\
&=1+\sum_{i=0}^{n-1}F(S_i-S_{n-1})\\
&=n+\frac12-\sum_{i=0}^{n-2}F(X_{i+1}+\dots+X_{n-1}),
\end{align*}
where the last step uses the symmetry of $F$. Note that this observation applies only to the full information version.
\end{enumerate}

\begin{proof}[Proof of Theorem \ref{thm:two-step}]
Let $n=2$, and consider first the full information case. If we take the first step and $S_1\leq 0$, we should continue by observation 2 above, and our expected rank is $\EE(R_2|S_1)=2.5-F(X_1)$ by observation 3 above. Suppose $S_1>0$ instead. Then $\EE(R_2|S_1)=2.5-F(X_1)\geq 1.5=\EE(R_1|S_1)$, so we should stop, with expected rank $1.5$. Thus, when we take at least the first step, the optimal expected rank is
\begin{equation*}
\frac12(1.5)+\int_{-\infty}^0 \{2.5-F(x)\}dF(x)=2-\int_0^{1/2} u du=\frac{15}{8}.
\end{equation*}
As this is less than 2 (the expected rank of $S_0$), we should take the first step and our optimal expected rank is $15/8$. This shows that the optimal rule is as stated in the theorem. Since $\tau^*$ uses only the relative rank of $S_1$ (that is, the comparison of $S_1$ to $0=S_0$), it follows that this rule is optimal in the relative ranks version as well.
\end{proof}

We next consider the case $n=3$. Here we use backward induction to determine the optimal rule. First, we define the quantities
\begin{gather}
V_i(X_1,\dots,X_i):=\inf_{i\leq\tau\leq 3} \EE(R_\tau|\FF_i), \qquad i=0,1,2, \label{eq:V}\\
W_i(X_1,\dots,X_i):=\inf_{i<\tau\leq 3} \EE(R_\tau|\FF_i), \qquad i=0,1,2, \label{eq:W}
\end{gather}
where in each case, the infimum is over the set of all stopping times relative to the filtration $\{\FF_i\}$ that take values in the specified range. In case $i=0$, we write simply $V_0$ and $W_0$ for the quantities on the left. We also denote $V_0$ by $V$, and note that $V=\inf_\tau \EE(R_\tau)$. Observe that
\begin{equation}
V_i(X_1,\dots,X_i)=\min\{\EE(R_i|\FF_i),W_i(X_1,\dots,X_i)\}, \qquad i=0,1,2,
\label{eq:stop-or-continue}
\end{equation}
and
\begin{equation}
W_i(X_1,\dots,X_i)=\int_{-\infty}^\infty V_{i+1}(X_1,\dots,X_i,x)\,dF(x), \qquad i=0,1.
\label{eq:backward-recursion}
\end{equation}

Starting one step before the end of the random walk, suppose $X_1$ and $X_2$ (or equivalently, $S_1$ and $S_2$) have been observed. If $\tilde{R}_2=3$, then the walk is at a running minimum, so it is optimal continue, and by observation 3,
\[
V_2(X_1,X_2)=W_2(X_1,X_2)=3.5-F(X_2)-F(X_1+X_2).
\]
If, on the other hand, $\tilde{R}_2=1$, then it is optimal to stop, since $\EE(R_2|\FF_2)=1.5\leq 3.5-F(X_2)-F(X_1+X_2)=W_2(X_1,X_2)$. So in this case, $V_2(X_1,X_2)=1.5$. 

The interesting case is when $\tilde{R}_2=2$. This can happen in two different ways: (i) $X_1>0$ and $-X_1<X_2<0$; or (ii) $X_1<0$ and $0<X_2<-X_1$. Note that in either case, $\EE(R_2|\FF_2)=2.5$, which must be compared to $W_2(X_1,X_2)=3.5-F(X_2)-F(X_1+X_2)=2.5+F(-X_2)-F(X_1+X_2)$. Thus it is optimal to stop if $F(-X_2)\geq F(X_1+X_2)$ and to continue otherwise; in other words, it is optimal to stop if $X_2\leq -X_1/2$, and to continue otherwise. (Note that there could be a region of points $(X_1,X_2)$ for which we are indifferent between stopping and continuing; this is why we do not say ``it is optimal to stop if and only if $X_2\leq -X_1/2$".)

Putting these observations together, we see that, if we had not yet stopped before, it is optimal to stop at time 2 if one of the following holds:
\begin{enumerate}
\item $X_1>0$ and $X_2>0$; or
\item $X_1>0$ and $-X_1<X_2\leq -X_1/2$; or
\item $X_1<0$ and $X_2>-X_1$; or
\item $X_1<0$ and $0<X_2\leq -X_1/2$;
\end{enumerate}
and to continue otherwise. Moreover,
\begin{equation}
V_2(X_1,X_2)=\min\{\tilde{R}_2+0.5,1.5-F(-X_2)-F(-X_1-X_2)\}.
\end{equation}

From the above facts, we can compute the optimal expected rank if we continue after the first step, that is, after observing $X_1=x$: For $x>0$, we have
\begin{align*}
W_1(x)&= 1.5 \PP(X_2 > 0) + \int_{-x/2}^0 \{1.5 + F(-x-y)+F(-y)\} dF(y) \\
& \hspace{1cm} + 2.5\PP\left(-x < X_2 < -\frac{x}{2}\right) + \int_{-\infty}^{-x} \{1.5 + F(-x-y) + F(-y)\} dF(y) \\ 
&= 1.5 + \int_{-x/2}^0 \{F(-x - y) + F(-y)\} dF(y) \\
& \hspace{1cm} + \PP\left(-x < X_2 < -\frac{x}{2}\right) + \int_{-\infty}^{-x} \{F(-x -y) + F(-y)\} dF(y).
\end{align*}
Using the symmetry of $f$, we can simplify this by noting that
\begin{equation}
\int_{-b}^{-a} F(-y)dF(y)=\int_a^b F(z)dF(z)=\frac12[F^2(b)-F^2(a)]
\label{eq:symmetry-in-integral1}
\end{equation}
and
\begin{equation}
\int_{-b}^{-a} F(-x-y)dF(y)=\int_a^b F(z-x)dF(z)
\label{eq:symmetry-in-integral2}
\end{equation}
for $0\leq a<b\leq\infty$; and by using also that $F(0)=1/2$. This gives, for $x>0$,
\begin{align}
\begin{split}
W_1(x)&= \frac{15}{8} + \int_0^{x/2} F(y-x)dF(y)+\int_x^\infty F(y-x)dF(y) \\
& \hspace{1cm} + F(x)-F\left(\frac{x}{2} \right) -\frac{1}{2}\left[F^2(x)-F^2 \left(\frac{x}{2}\right) \right].
\end{split}
\label{eq:W1-expression}
\end{align}
We must compare this with $\EE(R_1|\FF_1)=2$. Taking limits under the integral signs in \eqref{eq:W1-expression} (which is justified by dominated convergence), we obtain
\[
\lim_{x\to\infty}W_1(x)=\frac{15}{8}
\]
(which, not coincidentally, is equal to the value of the $2$-step problem); and
\[
\lim_{x\searrow 0}W_1(x)=\frac{15}{8}+\int_0^\infty F(y)dF(y)=\frac{15}{8}+\int_{1/2}^1 u du=\frac{9}{4}.
\]
Since $W_1(x)$ is clearly continuous, it follows that there is a critical point $x_1^*>0$ (not necessarily unique) such that $W_1(x_1^*)=2$. Since $W_1(x)$ is also nonincreasing in $x$ (an immediate consequence of the definition), we conclude that it is optimal to stop at time $1$ if $0<X_1\leq x_1^*$, and to continue if $X_1>x_1^*$. (Intuitively, if $X_1$ is very large, one should continue since the risk of falling back below the starting point of $0$ is very small.) Note that $x_1^*$ is a solution of the equation
\begin{align}
\begin{split}
\int_0^{x/2} F(y-x)dF(y) &+ \int_x^\infty F(y-x)dF(y) \\
& + F(x)-F\left(\frac{x}{2} \right)-\frac{1}{2}\left[F^2(x)-F^2 \left(\frac{x}{2}\right) \right]=\frac18.
\end{split}
\label{eq:x1-equation}
\end{align}

As shown below, $x_1^*$ is always a high quantile of the distribution of $X_1$.

\begin{lemma} \label{lem:quantile-estimate}
We have 
\[
F(x_1^*)\geq \frac12+\frac{\sqrt{2}}{4}\approx .85355.
\]
\end{lemma}

\begin{proof}
In \eqref{eq:W1-expression}, we use $F(y-x)\geq F(-x)$ in the first integral and $F(y-x)\geq F(0)=1/2$ in the second to obtain (using the symmetry of $F$),
\begin{align}
\begin{split}
W_1(x)&\geq \frac{15}{8}+F(x)-\frac12 F^2(x)-F(x)F\left(\frac{x}{2}\right)+\frac12 F^2\left(\frac{x}{2}\right)\\
&=\frac{15}{8}+F(x)-F^2(x)+\frac12\left[F(x)-F\left(\frac{x}{2}\right)\right]^2\\
&\geq \frac{15}{8}+F(x)\{1-F(x)\}.
\end{split}
\label{eq:W1-lower-bound}
\end{align}
Setting $W_1(x_1^*)=2$ thus yields $F(x_1^*)\{1-F(x_1^*)\}\leq 1/8$, and since we know $F(x_1^*)\geq F(0)=1/2$, it follows that $F(x_1^*)\geq \frac12+\frac{\sqrt{2}}{4}$.
\end{proof}

We now consider the case when $X_1=x<0$. Here it is always optimal to continue as we are at a running minimum, and
\begin{align*}
W_1(x)&= 1.5\PP(X_2>-x) + \int_{-x/2}^{-x} \{1.5 + F(-y) + F(-x - y)\} dF(y) \\ 
& \hspace{1cm} + 2.5 \PP(0 < X_2 < -x/2) + \int_{-\infty}^0 \{1.5 + F(-y) + F(-x - y)\} dF(y) \\
&= 1.5 + \int_{-x/2}^{-x} \{F(-y) + F(-x - y)\} dF(y) \\ 
& \hspace{1cm} + \PP(0 < X_2 < -x/2) + \int_{-\infty}^0 \{F(-y) + F(-x - y)\} dF(y).
\end{align*}
Using \eqref{eq:symmetry-in-integral1} and \eqref{eq:symmetry-in-integral2}, this leads to
\begin{align}
\begin{split}
W_1(x)&= \frac{19}{8} + \int_{x}^{x/2} F(y-x) dF(y) + \int_0^{\infty} F(y-x) dF(y)\\ 
& \hspace{1cm} - F\left(\frac{x}{2} \right) + \frac{1}{2} \left[F^2\left(\frac{x}{2}\right) - F^2(x) \right].
\end{split}
\label{eq:W1-expression-negative}
\end{align}

\begin{proof}[Proof of Theorem \ref{thm:three-step-full}]
We have already determined the optimal rule and shown, in Lemma \ref{lem:quantile-estimate}, that $F(x_1^*)\geq \frac12+\frac{\sqrt{2}}{4}$. It remains to prove the estimates \eqref{eq:V-estimates} and to show that the upper bound is attained.

First we estimate $W_1(x)$ by a simpler expression for $x<0$. In \eqref{eq:W1-expression-negative}, use $F(y-x)\geq F(0)=1/2$ in the first integral to get
\[
\int_x^{x/2} F(y-x)dF(y) \geq \frac12\left[F\left(\frac{x}{2}\right)-F(x)\right].
\]
The second integral we estimate as follows:
\begin{align*}
\int_0^{\infty} F(y-x) dF(y) &\geq \int_0^{-x}F(-x)dF(y)+\int_{-x}^\infty F(y)dF(y)\\
&=F(-x)\left\{F(-x)-\frac12\right\}+\frac12\left\{1-F^2(-x)\right\}\\
&=\frac12\big(1-F(-x)\{1-F(-x)\}\big)\\
&=\frac12\big(1-F(x)\{1-F(x)\}\big).
\end{align*}
Putting these estimates back into \eqref{eq:W1-expression-negative} yields
\begin{equation}
W_1(x)\geq \frac{23}{8}-F(x)-\frac12 F\left(\frac{x}{2}\right)+\frac12 F^2\left(\frac{x}{2}\right).
\label{eq:W1-lower-estimate}
\end{equation}
Note that $V_1(x)=2$ if $0<x\leq x_1^*$, and $V_1(x)=W_1(x)$ otherwise.
At time $0$ it is optimal to continue since the walk is at a minimum. Hence,
\begin{align}
V&=\int_{-\infty}^\infty V_1(x)dF(x) \notag \\
&=\int_{-\infty}^0 W_1(x)dF(x)+2\PP(0<X_1\leq x_1^*)+\int_{x_1^*}^\infty W_1(x)dF(x). \label{eq:three-terms}
\end{align}
We estimate each integral separately. First, by \eqref{eq:W1-lower-estimate},
\begin{align*}
\int_{-\infty}^0 W_1(x)dF(x) &\geq \frac{23}{16}-\int_{-\infty}^0 F(x)dF(x)
-\frac12 \int_{-\infty}^0 F\left(\frac{x}{2}\right)\left\{1-F\left(\frac{x}{2}\right)\right\}dF(x)\\
&\geq \frac{23}{16}-\int_0^{1/2}u du-\frac18 \int_{-\infty}^0 dF(x)
=\frac{5}{4}.
\end{align*}
The other integral can be estimated below using \eqref{eq:W1-lower-bound}, which gives
\begin{align*}
\int_{x_1^*}^\infty W_1(x)dF(x) &\geq \int_{x_1^*}^\infty \left[\frac{15}{8}+F(x)\{1-F(x)\}\right]dF(x)\\
&=\frac{15}{8}\{1-F(x_1^*)\}+\frac16-\frac12 F^2(x_1^*)+\frac13 F^3(x_1^*).
\end{align*}
Combining this with the second term in \eqref{eq:three-terms} yields
\begin{align*}
\int_0^\infty V_1(x)dF(x)&=2\left(F(x_1^*)-\frac12\right)+\int_{x_1^*}^\infty W_1(x)dF(x)\\
&\geq \frac{25}{24}+\frac{1}{24}\left\{3F(x_1^*)-12F^2(x_1^*)+8F^3(x_1^*)\right\}\\
&\geq \frac{25}{24}-\frac{1+\sqrt{2}}{48}.
\end{align*}
To see the last inequality, let $g(t)=3t-12t^2+8t^3$ and note that $g$ has a local minimum value of $-\frac{1+\sqrt{2}}{2}$ at $t=\frac12+\frac{\sqrt{2}}{4}$, so the cubic polynomial in $F(x_1^*)$ is minimized exactly when $F(x_1^*)=\frac12+\frac{\sqrt{2}}{4}$.

Combining the estimates, we finally arrive at
\[
V\geq \frac{5}{4}+\frac{25}{24}-\frac{1+\sqrt{2}}{48}=\frac{109-\sqrt{2}}{48}\approx 2.24137.
\]

In the next section we will show that $55/24$ is a sharp upper bound for $\EE(R_{\tau^*})$ in the relative ranks version of the problem. Since in the full information version we can do at least as well, it follows that $V\leq 55/24$. This is attained (in both versions of the problem) when $X$ has the uniform distribution on $(-2,-1)\cup(1,2)$; we leave the details to the interested reader. Thus, the proof is complete.
\end{proof}

\begin{examples}
{\rm
{\em (a)} Let $X$ have the uniform distribution on $(-1,1)$. Then for $0<x<1$ we have
\[
W_1(x)=\frac{9}{4} - \frac{x}{4} - \frac{x^2}{16},
\]
so that $x_1^*=2(\sqrt{2}-1) \approx 0.828$; and for $-1<x<0$ we have
\[
W_1(x)=\frac{9}{4} - \frac{3x}{4} - \frac{3x^2}{16}.
\]
The optimal expected rank is
\begin{equation*}
V=\frac12\int_{-1}^0 W_1(x)dx+\frac12\cdot 2x_1^*+\frac12\int_{x_1^*}^1 W_1(x)dx=\frac{11}{4}-\frac{\sqrt{2}}{3}\approx 2.279.
\end{equation*}

{\em (b)} Let $X$ have the standard two-sided exponential (or Laplace) distribution, with density $f(x)=\frac12 e^{-|x|}$ for $x\in\RR$. Then for $x>0$ we have
\[
W_1(x)=\frac{15}{8} + \frac{1}{8} xe^{-x} + \frac{1}{2} e^{-x} - \frac{1}{8} e^{-2x},
\]
and numerically solving $W_1(x)=2$ gives $x_1^*\approx 1.71$. For $x<0$ we have
\[
W_1(x)=\frac{23}{8} + \frac{1}{8} x e^{x} - \frac{1}{2} e^{x} - \frac{1}{8} e^{2x}.
\]
The optimal expected rank is approximately $2.271$.
}
\end{examples}

\section{Proofs for the relative ranks case} \label{sec:proofs-relative}

In this section we assume that only the relative ranks $\tilde{R}_i$, $i=0,1,\dots,n$ are observed. We fix $n=3$. Before deriving the optimal rule, we prove Proposition \ref{prop:p-bounds-unimodal}. In what follows, $|X|_{(1)},|X|_{(2)}$ and $|X|_{(3)}$ denote the order statistics of $|X_1|,|X_2|$ and $|X_3|$, so $(|X|_{(1)},|X|_{(2)},|X|_{(3)})$ is a permutation of $(|X_1|,|X_2|,|X_3|)$ with $|X|_{(1)}\leq |X|_{(2)}\leq |X|_{(3)}$. We let $G$ denote the distribution function of $|X_1|$, so $G(x)=2F(x)-1$ for $x\geq 0$.

\begin{proof}[Proof of Proposition \ref{prop:p-bounds-unimodal}]
It is easy to see, for {\em any} continuous symmetric distribution, that $p>0$. For instance, choose $x_0>0$ so that $F(2x_0)>F(x_0)$; such a point certainly exists. Then, by symmetry (since $X_1,X_2,X_3$ are i.i.d.),
\begin{align*}
48p&=\PP\big(|X|_{(3)}<|X|_{(1)}+|X|_{(2)}\big) \geq \PP(x_0<X_i\leq 2x_0\ \mbox{for}\ i=1,2,3)\\
&=\big(F(2x_0)-F(x_0)\big)^3>0.
\end{align*}
Now assume $X\in\mathscr{U}$; this implies $G(x+y)\leq G(x)+G(y)$ for all $x,y\geq 0$. We calculate
\begin{align*}
16q&=\frac{1}{3}\PP\big(|X|_{(3)}>|X|_{(1)}+|X|_{(2)}\big)=\PP\big(|X_3|>|X_1|+|X_2|\big)\\
&=\int_0^\infty \int_0^\infty\{1-G(x+y)\}dG(x)dG(y)\\
&\geq\int_0^\infty \int_0^\infty{\{1-G(x)-G(y)\}}^+dG(x)dG(y)\\
&=\int_0^1 \int_0^1 (1-u-v)^+ du dv=\frac{1}{6}.
\end{align*}
The only inequality in this calculation becomes an equality when $X$ is uniformly distributed on $(-1,1)$. Thus, $p\leq 1/96$, and this bound is attained for the uniform distribution on $(-1,1)$. 

It remains to show that the lower bound $p>0$ is sharp. To this end, let $G(x)=x^\delta$ for $0\leq x\leq 1$ and $\delta>0$. (This corresponds with a density $f(x)=\frac{\delta}{2}|x|^{\delta-1}$ for $x\in(-1,1)\backslash\{0\}$.) For $0\leq u\leq v\leq 1$, we have
\[
G\big(G^{-1}(u)+G^{-1}(v)\big)\leq G\big(2G^{-1}(v)\big)=G\big(2v^{1/\delta}\big)\leq 2^\delta v.
\]
Hence, for this $G$,
\begin{align*}
16q&=\int_0^\infty \int_0^\infty\{1-G(x+y)\}dG(x)dG(y)\\
&=\int_0^1 \int_0^1 \left\{1-G\big(G^{-1}(u)+G^{-1}(v)\big)\right\}du dv\\
&\geq \int_0^1 \int_0^1 \left(1-2^\delta\max\{u,v\}\right)du dv\\
&\to \int_0^1 \int_0^1 (1-\max\{u,v\})du dv=\frac13 \qquad\mbox{as $\delta\to 0$}.
\end{align*}
Thus $q$ gets arbitrarily close to $1/48$, and $p=(1/48)-q$ gets arbitrarily close to $0$, for sufficiently small $\delta>0$.
\end{proof}

We will use the parameters $p$ and $q$ to express the probabilities of all 24 possible rank orderings of $S_0,S_1,S_2$ and $S_3$. Let
\[
\Delta:=\{|X|_{(3)}<|X|_{(1)}+|X|_{(2)}\}
\]
be the event that the absolute step sizes satisfy the triangle inequality, and note that $\PP(\Delta)=48p$. Define the $\sigma$-algebras
\begin{gather*}
\mathscr{D}:=\sigma(\Delta), \qquad \mathscr{S}:=\sigma(\{\sign(X_i): i=1,2,3\}),\\
\mathscr{C}:=\sigma(\{\sign(|X_i|-|X_j|): i,j=1,2,3, i\neq j\}).
\end{gather*}
That is, $\mathscr{D}$ is the $\sigma$-algebra generated by $\Delta$, $\mathscr{S}$ is the $\sigma$-algebra generated by the signs of $X_1,X_2$ and $X_3$, and $\mathscr{C}$ is the $\sigma$-algebra generated by the mutual comparisons of $|X_1|,|X_2|$ and $|X_3|$. Observe that the $\sigma$-algebras $\mathscr{C},\mathscr{D}$ and $\mathscr{S}$ are independent by the symmetry of $X$. This makes it easy to calculate the probabilities of the 24 permutations of the random walk. For example,
\begin{align*}
\PP(0>S_3 &> S_2 > S_1) \\
&= \PP\big(\{X_1<0, X_2 > 0, X_3 > 0\} \cap \{|X_1| > |X_2|, |X_1| > |X_3|\} \cap \Delta^{c}\big) \\
&= \PP(X_1<0, X_2 > 0, X_3 > 0)\, \PP(|X|_{(3)}=|X_1|)\, \PP(\Delta^{c}) \\
&= \frac18 \cdot \frac13 \cdot 48q = 2q.
\end{align*}
The other probabilities can be derived similarly; we list them in Table \ref{tab:permutations}.

\begin{table}
\begin{center}
\renewcommand{\arraystretch}{1.25}
\begin{tabular}{|c|c|c|}
\hline
Permutation & Reflection & Probability \\ \hline
$0>S_1 > S_2 > S_3$ & $0<S_1<S_2<S_3$ & $1/8$\\
$0>S_1 > S_3 > S_2$ & $0<S_1<S_3<S_2$ & $1/16$\\
$0>S_2 > S_1 > S_3$ & $0<S_2<S_1<S_3$ & $1/24$\\
$0>S_2 > S_3 > S_1$ & $0<S_2<S_3<S_1$ & $1/48$\\
$0>S_3 > S_1 > S_2$ & $0<S_3<S_1<S_2$ & $(1/48)+2p$\\
$0>S_3 > S_2 > S_1$ & $0<S_3<S_2<S_1$ & $2q$\\
$S_2 > 0 > S_1 > S_3$ & $S_2<0<S_1<S_3$ & $1/48$\\
$S_2 > 0 > S_3 > S_1$ & $S_2<0<S_3<S_1$ & $2p$\\
$S_2 > S_3 > 0 > S_1$ & $S_2<S_3<0<S_1$ & $2q$\\
$S_3 > 0 > S_1 > S_2$ & $S_3<0<S_1<S_2$ & $2q$\\
$S_3 > 0 > S_2 > S_1$ & $S_3<0<S_2<S_1$ & $(1/48)+2p$\\
$S_3 > S_2 > 0 > S_1$ & $S_3<S_2<0<S_1$ & $1/16$\\
\hline
\end{tabular}
\caption{The probabilities of the 24 possible rank orderings of $0=S_0,S_1,S_2$ and $S_3$. The first column lists permutations with $S_1<0$; the second lists the ones with $S_1>0$. By symmetry, the two permutations in each row have the same probability.} 
\label{tab:permutations}
\end{center}
\end{table}

\begin{proof}[Proof of Theorem \ref{thm:three-step-relative-ranks}]
Recall the filtration $\{\GG_i\}$ defined by $\GG_i=\sigma(\tilde{R}_1,\dots,\tilde{R}_i)$, $i=0,1,2,3$. Analogously to \eqref{eq:V} and \eqref{eq:W} we define
\begin{gather*}
\tilde{V}_i(\tilde{R}_1,\dots,\tilde{R}_i):=\inf_{i\leq\tau\leq 3} \EE(R_\tau|\GG_i), \qquad i=0,1,2,\\
\tilde{W}_i(\tilde{R}_1,\dots,\tilde{R}_i):=\inf_{i<\tau\leq 3} \EE(R_\tau|\GG_i), \qquad i=0,1,2.
\end{gather*}
In case $i=0$ we write simply $\tilde{V}_0$ and $\tilde{W}_0$ for the left hand sides, and we denote $\tilde{V}_0$ also by $\tilde{V}$. As in the full information case (cf.~\eqref{eq:stop-or-continue}),
\[
\tilde{V}_i(\tilde{R}_1,\dots,\tilde{R}_i)=\min\{\EE(R_i|\GG_i),\tilde{W}_i(\tilde{R}_1,\dots,\tilde{R}_i)\}.
\]

\medskip
{\em a)} Assume first that $p\leq q$; recall that this is the case for all unimodal distributions. Suppose $\tilde{R}_1$ and $\tilde{R}_2$ have been observed. Equivalently, the mutual comparisons between $0=S_0$, $S_1$ and $S_2$ are known. We consider the six possible permutations one by one:

\bigskip
{\em Case 1.} Suppose $0 < S_1 < S_2$. The probability of this event is $1/4$. Here $\EE(R_2|\GG_2)=1.5\leq\EE(R_3|\GG_2)$, since $R_3$ takes the value $1$ with (conditional) probability $1/2$, and otherwise takes at least the value $2$.
Thus, it is optimal to stop, and $\tilde{V}_2(\tilde{R}_1,\tilde{R}_2)=1.5$.

\bigskip
{\em Case 2.} Suppose $0 < S_2 < S_1$. This happens with probability $1/8$. Here $\EE(R_2|\GG_2)=2.5$, whereas
\[
\EE(R_3|\GG_2)=8\left[1\cdot \frac{1}{24}+2\cdot \frac{1}{48}+3\cdot 2q+4\cdot \left(2p+\frac{1}{48}\right)\right]=\frac{7}{3}+16p, 
\]
where we used the third, fourth, sixth and eleventh rows of Table \ref{tab:permutations}. Since $p\leq 1/96$,  it follows that $\EE(R_3|\GG_2)\leq 2.5$. Hence, it is optimal to continue, and $\tilde{V}_2(\tilde{R}_1,\tilde{R}_2)=\frac{7}{3}+16p$. 

\bigskip
{\em Case 3.} Suppose $S_2 < 0 < S_1$. This happens with probability $1/8$. Here the walk is at a minimum, so it is optimal to continue, and
\[
\tilde{V}_2(\tilde{R}_1,\tilde{R}_2)=\EE(R_3|\GG_2)=8\left[1\cdot \frac{1}{48}+2\cdot 2p+3\cdot 2q+4\cdot \frac{1}{16}\right]=\frac{35}{12}+16q.
\]

\medskip
{\em Case 4.} Suppose $S_1 < 0 < S_2$. This happens with probability $1/8$. Here the walk is at a maximum, so as in Case 1 it is optimal to stop, and $\tilde{V}_2(\tilde{R}_1,\tilde{R}_2)=1.5$.

\bigskip
{\em Case 5.} Suppose $S_1 < S_2 < 0$. This happens with probability $1/8$. Here $\EE(R_2|\GG_2)=2.5$, whereas
\[
\EE(R_3|\GG_2)=8\left[1\cdot \left(2p+\frac{1}{48}\right)+2\cdot 2q+3\cdot \frac{1}{48}+4\cdot \frac{1}{24}\right]=\frac{7}{3}+16q\geq 2.5,
\]
using that $q\geq 1/96$. Thus, it is optimal to stop, and $\tilde{V}_2(\tilde{R}_1,\tilde{R}_2)=2.5$.

\bigskip
{\em Case 6.} Suppose $S_2 < S_1 < 0$. This occurs with probability $1/4$. Since we are at a minimum, it is optimal to continue and
\[
\tilde{V}_2(\tilde{R}_1,\tilde{R}_2)=\EE(R_3|\GG_2)=4\left[1\cdot 2q+2\cdot \left(2p+\frac{1}{48}\right)+3\cdot \frac{1}{16}+4\cdot \frac{1}{8}\right]=\frac{37}{12} + 8p.
\]

This completes the analysis of the situation after two steps. 

We assume next that $\tilde{R}_1$ has been observed; that is, we know whether $S_1>0$ or $S_1<0$. If $S_1>0$, then $\EE(R_1|\GG_1)=2$, whereas
\begin{align*}
\tilde{W}_1(\tilde{R}_1)&=\sum_{k=1}^3 \PP(\tilde{R}_2=k|\GG_1)\tilde{V}_2(\tilde{R}_1,k)\\
&=\PP(S_2>S_1>0|S_1>0)\cdot (1.5)+\PP(S_1>S_2>0|S_1>0)\cdot \left(\frac{7}{3}+16p\right)\\
& \hspace{2cm} +\PP(S_1>0>S_2|S_1>0)\cdot \left(\frac{35}{12}+16q\right)\\
&=\frac12\cdot (1.5)+\frac14\left(\frac{7}{3}+16p\right)+\frac14\left(\frac{35}{12}+16q\right)=\frac{103}{48}>2,
\end{align*}
where we used the results from Cases 1-3 above. Thus, it is optimal to stop, and $\tilde{V}_1(\tilde{R}_1)=2$.

On the other hand, if $S_1<0$, then the walk is at a minimum, and so
\begin{align*}
\tilde{V}_1(\tilde{R}_1)&=\tilde{W}_1(\tilde{R}_1)=\sum_{k=1}^3 \PP(\tilde{R}_2=k|\GG_1)\tilde{V}_2(\tilde{R}_1,k)\\
&=\PP(S_2>0>S_1|S_1<0)\cdot (1.5)+\PP(0>S_2>S_1|S_1<0)\cdot (2.5)\\
& \hspace{2cm} +\PP(0>S_1>S_2|S_1<0)\cdot \left(\frac{37}{12} + 8p\right)\\
&=\frac14 \cdot (1.5)+\frac14 \cdot (2.5)+\frac12 \cdot \left(\frac{37}{12} + 8p\right)=\frac{61}{24} + 4p,
\end{align*}
using the results from Cases 4-6 above. We thus obtain
\begin{align*}
\tilde{V}&=\tilde{W}_0=\PP(S_1>0)\tilde{V}_1(1)+\PP(S_1<0)\tilde{V}_1(2)\\
&=\frac12\cdot 2+\frac12\cdot \left(\frac{61}{24} + 4p\right)=\frac{109}{48} + 2p.
\end{align*}
Since $0<p\leq 1/96$ we have $109/48<\tilde{V}\leq 55/24$, and both bounds are sharp.

\bigskip
{\em b)} Now suppose $p\geq q$. The analysis being very similar, we only indicate at which points it differs from the preceding case. We focus first on the situation after two steps. Note that in Case 2 it is now optimal to stop, with expected rank 2.5. On the other hand, in Case 5 it now becomes optimal to continue, with expected rank $\frac{7}{3}+16q$. The calculation of $\tilde{V}_1(\tilde{R}_1)$ changes as follows: If $S_1>0$, we now obtain 
\[
\tilde{W}_1(\tilde{R}_1)=\frac{101}{48} + 4q.
\]
This is still greater than $2=\EE(R_1|\GG_1)$, so it remains optimal to stop, and $\tilde{V}_1(\tilde{R}_1)=2$. On the other hand, if $S_1<0$ then we get $\tilde{V}_1(\tilde{R}_1)=\tilde{W}_1(\tilde{R}_1)=31/12$. The optimal expected rank is thus
$\tilde{V}=\frac12\cdot 2+\frac12\cdot \frac{31}{12}=\frac{55}{24}$.

Finally, a close examination of the above analysis reveals that the optimal rule is as stated in the theorem.
\end{proof}

\begin{example}
{\rm
For the two-sided exponential (or Laplace) distribution, we have $p = 1/192$ and $\tilde{V} = 73/32 = 2.28125$. For the Uniform$(-1,1)$ distribution, $p=1/96$ and $\tilde{V}=55/24\approx 2.2917$. We summarize the numerical results of this paper in Table~\ref{tab:summary}.
}
\end{example}

\begin{table}
\begin{center} 
\renewcommand{\arraystretch}{1.25}
\begin{tabular}{| c | c | r |}
\hline
{\bf Version} & {\bf Description} & {\bf Expected Rank} \\ \hline
Full Information & Lower bound & 2.2413 \\ \hline
Full Information & Laplace Distribution & $\approx 2.271$ \\ \hline
Full Information & Uniform Distribution & $\approx 2.279$ \\ \hline
Full Information & Maximum & $\frac{55}{24} = 2.291\overline{6}$ \\ \hline
Relative Ranks & Greatest Lower Bound & $\frac{109}{48} \approx 2.2708$ \\ \hline
Relative Ranks & Laplace Distribution & $\frac{73}{32} = 2.28125$ \\ \hline
Relative Ranks & Uniform Distribution & $\frac{55}{24} = 2.291\overline{6}$ \\ \hline
Relative Ranks & Maximum & $\frac{55}{24} = 2.291\overline{6}$ \\ \hline
Both Versions & Stopping Immediately & 2.5 \\ \hline
\end{tabular}
\caption{Summary of results for $n=3$}
\label{tab:summary}
\end{center}
\end{table}

\section{Concluding remarks}

We have derived the optimal stopping rules for $n=2$ and $n=3$, in both the full information version and the relative rank version of the problem. When $n=3$, both the optimal rule and the optimal expected rank depend on the distribution of the step sizes, though less so in the relative rank version. 

For the full information version of the problem it seems unlikely, in light of the complexity of the optimal rule already for $n=3$, that the problem can be solved exactly for even moderately large values of $n$. In the relative ranks version, the exact solution can probably be found for a few larger values of $n$, though we have not attempted to do so. One difficulty is that, while for $n=3$ precisely half of the probabilities of the 24 permutations did not depend on the distribution of $X$, this proportion seems to decrease rapidly as $n$ grows larger. Moreover, the probabilities that depend on the distribution can do so in more complicated ways. Even when $n=4$, for instance, we might need to consider probabilities such as $\PP(|X|_{(4)}<|X|_{(1)}+|X|_{(2)}+|X|_{(3)})$, $\PP(|X|_{(4)}<|X|_{(1)}+|X|_{(2)})$, $\PP(|X|_{(4)}+|X|_{(1)}<|X|_{(3)}+|X|_{(2)})$, etc. That said, since there are essentially only finitely many different stopping rules to consider, a ``brute force" computer algorithm could in principle come up with the optimal rule as long as $n$ is not too large.

A more interesting approach, however, would be to develop relatively simple stopping rules which perform well asymptotically for large $n$, and to aim for reasonably sharp upper and lower bounds on the ratio $V^{(n)}/n$, where $V^{(n)}$ denotes the optimal expected rank for an $n$-step problem. This will be the subject of a forthcoming paper.

\end{document}